\documentclass[a4paper,reqno]{amsart}
\usepackage{amsfonts}
\usepackage{amssymb}
\usepackage{times}
\usepackage{xfrac}
 \usepackage{cancel}

\usepackage{tasks}
\usepackage{cite}
\usepackage{fullpage}

\newcommand{\CC}{\mathbb{C}}
\newcommand{\KK}{\mathbb{K}}

\def\a{{\alpha}}

\def\b{{\beta}}

\def\d{{\delta}}

\def\g{{\gamma}}

\def\l{{\lambda}}

\def\gg{{\mathfrak g}}
\def\cal{\mathcal }

\makeatletter
\def\Ddots{\mathinner{\mkern1mu\raise\p@
\vbox{\kern7\p@\hbox{.}}\mkern2mu
\raise4\p@\hbox{.}\mkern2mu\raise7\p@\hbox{.}\mkern1mu}}
\makeatother

\theoremstyle{plain}%

\newtheorem{thm}{Theorem}[section]
 
 \newtheorem{lem}[thm]{Lemma}
  \newtheorem{defn}[thm]{Definition}
 \newtheorem{exa}[thm]{Example}
 \newtheorem{prop}[thm]{Proposition}
  
  \newtheorem{rem}[thm]{Remark}
  \newtheorem{conjecture}[thm]{Conjecture}
  
\newfont{\hueca}{msbm10}

\begin{document}

% MATH -------------------------------------------------------------------

\def\a{{\alpha}}

\def\b{{\beta}}

\def\d{{\delta}}

\def\g{{\gamma}}

\def\l{{\lambda}}

\def\gg{{\mathfrak g}}
\def\cal{\mathcal}

\title{ Solvable  compatible Lie   algebras with a given nilradical}

\author{A.  Fern\'andez Ouaridi}
\address{Amir Fern\'andez Ouaridi
\newline \indent Departament of Geometry and Topology, University of Sevilla, Sevilla (Spain)}
 \email{amir.fernandez.ouaridi@gmail.com}

\author{R.M. Navarro}
\address{Rosa Mar{\'\i}a Navarro \newline \indent
Dpto. de Matem{\'a}ticas, Universidad de Extremadura, C{\'a}ceres
 (Spain) }
\email{rnavarro@unex.es}
\thanks{The authors are partially supported by Uni\'on Europea,  Fondo Europeo de Desarrollo Regional and by Junta de Extremadura, la Autoridad de Gestión y el Ministerio de Hacienda, through project GR24068, and also  by Agencia Estatal de Investigaci\'on (Spain), grant PID2024-155502NB-I00 (European FEDER support included, EU)}

\author{B.A. Omirov}
\address{Bakhrom A. Omirov \newline \indent 
Institute for Advanced Study in Mathematics, Harbin, China, Institute of Technology, Harbin 150001 \newline \indent
Suzhou Research Institute, Harbin Institute of Technology, Harbin 215104, Suzhou, China}
\email{omirovb@mail.ru}

\author{G.O. Solijanova}
\address{Gulkhayo O. Solijanova \newline \indent
National University of Uzbekistan, 100174, Tashkent, Uzbekistan, \newline \indent
V.I. Romanovskiy Institute of Mathematics, Uzbekistan Academy of Sciences, Uzbekistan}
\email{gulhayo.solijonova@mail.ru}

\begin{abstract}
We extend the classical construction of solvable Lie algebras from a nilradical to compatible Lie algebras. Since the sum of nilpotent ideals may fail to be nilpotent, we replace the usual nilradical by a \emph{special nilradical} that behaves well with the mixed Jacobi identity. We use the maximal tori of diagonal derivations to build solvable extensions. The method is applied to the pairs $(\mathrm L_n,\mathrm R_n)$ and $(\mathrm L_n,\mathrm W_n)$, yielding explicit one-dimensional solvable extensions and proving nonexistence of higher-dimensional ones in these cases. 
We also study filiform compatible Lie algebras. We introduce the model family $\mathcal L_s$ and show that each $\mathcal L_s$ is a linear deformation of the model filiform Lie algebra $\mathcal L_k$. Finally, we study the existence of solvable extensions of this family, within the framework developed above.
\end{abstract}

\maketitle


\section{Introduction}

The study of compatible algebraic structures, wherein a single vector space is endowed with two distinct product operations such that any of their linear combinations preserves the underlying algebraic structure, has become an area of significant interest in both mathematics and mathematical physics. Compatible Lie algebras, in particular, have been shown to have deep connections with various fields, including the classical Yang-Baxter equation, bi-Hamiltonian structures in integrable systems, and the theory of principal chiral models. This has motivated recent research into their fundamental properties, from their operadic characterization to their cohomology and deformation theories.

Recent efforts have focused on the algebraic classification of these algebras. A significant line of inquiry has been the classification of 
nilpotent compatible Lie algebras, which has been systematically addressed using generalizations of the Skjelbred-Sund method to provide a complete classification in low dimensions \cite{compatible Lie}. This approach, which relies on constructing nilpotent algebras as central extensions of smaller ones, has proven highly effective for understanding the nilpotent class.

A natural next step after nilpotent algebras is the study of solvable algebras. In the classical theory of Lie algebras, a fundamental method for constructing and classifying solvable Lie algebras is to first define their maximal nilpotent ideal—the nilradical—and then construct the solvable algebra as an extension of it. This approach, often known as Mubarakzyanov's method, depends on the existence of non-nilpotent derivations of the nilradical to build non-nilpotent solvable extensions. See for instance \cite{Casas1,Ancochea,Casas11,Casas13,Casas15, LAA2010,lisa,LMA, JAlgebra, 2}

Thus, the present manuscript aims to extend this construction methodology to  compatible Lie algebras. However, this transition is not straightforward, as foundational concepts require careful adaptation. For instance, the sum of two nilpotent ideals in a compatible Lie algebra is not necessarily nilpotent, which complicates the traditional definition of the nilradical. To address this, the key step is to replace the usual nilradical by the \emph{special nilradical}, which behaves well under the mixed Jacobi identity. We study the basic compatibility constraints for  semidirect products and use maximal tori of diagonal derivations to build solvable extensions in this new setting. Building upon this framework, this work explores the construction of solvable compatible Lie algebras whose nilradicals are well-known filiform Lie algebras, such as $\mathrm{L}_n$, $\mathrm{R}_n$ and $\mathrm{W}_n$, thereby providing a systematic method for generating new classes of solvable compatible Lie algebras. The method is worked out for the pairs $(\mathrm L_n,\mathrm R_n)$ and $(\mathrm L_n,\mathrm W_n)$ of compatible Lie algebras, yielding explicit one dimensional solvable extensions and nonexistence of higher–dimensional extensions. 
Additionally, we show a characterization of filiform compatible Lie algebras and propose their realization as linear deformation of model filiform ones. 
For filiform compatible Lie algebras we introduce the model family $\mathcal L_s$ and show that each $\mathcal L_s$ is a linear deformation of the model filiform Lie algebra $\mathcal L_k$. We conclude the paper studying the existence of solvable extensions of this family of model filiform compatible Lie algebras.

\medskip

Throughout the present paper we will consider vector spaces and algebras over the field complex numbers. We also asumme that the reader is familiarized with the basics of  Lie theory.

\section{Preliminaries on Compatible Lie Algebras}

In this section, we recall the definition of compatible Lie algebras and some of their basic properties. Throughout, let $\KK$ be a field of characteristic not equal to $2$. Unless otherwise stated, all vector spaces, Lie algebras, and linear maps are assumed to be defined over $\KK$. 
First, we shall recall the following well-known result.

\begin{prop} \label{compatible}
Let ${\gg}_1=(\gg, [-,-]_1)$ and ${\gg}_2=(\gg, [-,-]_2)$ be Lie algebras over the vector space $\gg$. Then the following conditions are equivalent:
\begin{enumerate}
    \item $(\gg, [-,-])$ is a Lie algebra, where $[x,y]=[x,y]_1+[x,y]_2$ for all $x, y \in \gg$.

    \item $(\gg, [-,-]_{\lambda_1, \lambda_2})$ is a Lie algebra for all $\lambda_1$, $\lambda_2 \in \KK$, where $[x,y]_{\lambda_1, \lambda_2}=\lambda_1 [x,y]_1+\lambda_2 [x,y]_2.$ 

    \item The following identity (called the mixed Jacobi identity) holds for all $x, y, z \in \gg$
    $$L(x,y,z):=[[x,y]_1,z]_2+[[y,z]_1,x]_2+[[z,x]_1,y]_2+[[x,y]_2,z]_1+[[y,z]_2,x]_1+[[z,x]_2,y]_1=0.$$
\end{enumerate}    
\end{prop}

Thus, we have the definition of a compatible Lie algebra.

\begin{defn} 
    A compatible Lie algebra is a dialgebra $(\gg, [-,-]_1, [-,-]_2)$, where the pairs ${\gg}_1=(\gg, [-,-]_1)$ and ${\gg}_2=(\gg, [-,-]_2)$ are Lie algebras satisfying any of the three equivalent conditions in Proposition~\ref{compatible}.
\end{defn}

Given a Lie algebra ${\gg}_0=(\gg,[-,-]_0)$, the classical example of a compatible Lie algebra arises from a linear deformation of ${\gg}_0$.  
Precisely, if $\varphi:\gg\times\gg\to\gg$ is a $2$-cocycle of ${\gg}_0$, then for each $t\in\KK$ the bracket  
$$
[x,y]_t=[x,y]_0+t\,\varphi(x,y), \qquad x,y\in\gg,
$$
defines a Lie algebra structure on $\gg$.  
In this situation, the triple $(\gg,[-,-]_0, \varphi)$ is a compatible Lie algebra if and only if $\varphi$ satisfies the Jacobi identity (this condition is usually denoted as $\varphi\circ \varphi = 0$) if and only if $[x,y]_t$ satisfies the Jacobi identity for some $t \in \KK$.  
This construction provides a large class of examples which we discuss later in this paper. 
It should be noted that any Lie algebra is in particular a compatible Lie algebra.

Recall that a {\em subalgebra} of a compatible Lie algebra $\gg$ is a vector subspace of $\gg$ which is closed for both multiplications. Likewise, an {\em ideal} $\mathfrak{i}$ of a compatible Lie algebra $\gg$ is a vector subspace such that $[\mathfrak{i}, \gg]_1$, $[\mathfrak{i}, \gg]_2 \subseteq \mathfrak{i}$. In particular, the {\em center} of a compatible Lie algebra $\gg$, denoted by $Z(\gg)$, is the following ideal
$$Z(\gg)=\{ x \in \gg  \ | \ [x,\gg]_1=[x,\gg]_2=0\}=Z(\gg_1) \cap Z(\gg_2).$$

The notions of nilpotency and solvability of a compatible Lie algebra are defined as follow \cite{compatible Lie}. Given subspaces $\mathfrak{s}$ and $\mathfrak{t}$ of a compatible Lie algebra $\gg$.  The commutator of $\mathfrak{s}$ and $\mathfrak{t}$ is the vector space $[\mathfrak{s}, \mathfrak{t}]:=[\mathfrak{s}, \mathfrak{t}]_1+ [\mathfrak{s}, \mathfrak{t}]_2$.
Following this notation, the  {\em descending central sequence} of a compatible Lie algebra $\gg$ is defined in the same way as for Lie algebras ${\cal C}^0(\gg): =\gg$, ${\cal C}^{k+1}(\gg):=[\gg,{\cal C}^k(\gg)]$  for all $k\geq 0$. Consequently, if ${\cal C}^k(\gg)=\{0\}$ for some $k$, then the compatible Lie algebra is called  {\em nilpotent}. By induction, it can be proven that ${\cal C}^i(\gg)$ is an ideal of $\gg$ and $[\cal C^i(\gg), \cal C^j(\gg)]\subseteq \cal C^{i+j}(\gg)$ for any $i, j \in \mathbb{N}$. Also, any nilpotent compatible Lie algebra has non-zero center and $\mathfrak g$ is nilpotent if and only if $\mathfrak g/Z(\mathfrak g)$ is nilpotent. 

Likewise, we can define the {\em derived sequence} of $\gg$ as ${\cal D}^0(\gg): =\gg, \ \ {\cal D}^{k+1}(\gg):=[{\cal D}^k(\gg),{\cal D}^k(\gg)] \ \ \mbox{ for all }k\geq 0.$  If this sequence is stabilized in zero, then the compatible algebra is said to be {\em solvable}. Let us note that, as usual, all nilpotent compatible Lie algebras are solvable compatible Lie algebras. Also, it is clear that if $\gg$ is a solvable (resp. nilpotent) compatible Lie algebra, then both $\gg_1$ and $\gg_2$ are solvable (resp. nilpotent) Lie algebras, but the converse does not hold. Therefore, 
Engel's and Lie's Theorem remain valid for $\gg_1$ and $\gg_2$. 

The notion of derivation of a compatible Lie algebra is defined in \cite{compatible Lie} as a pair 
$(d_1,d_2)$ where $d_1$ is a derivation of $\gg_1$, $d_2$ is a derivation of $\gg_2$ and $d_1+d_2$ is a derivation of $(\gg, [[-.-]]=[-,-]_1 + [-,-]_2)$. However, we shall use the definition of derivation of a compatible Lie algebra $\gg$ as a derivation of a dialgebra $\gg$, that is, $d$ is called {\it a derivation of compatible Lie algebra $\gg$} if $d$ is a derivation of both Lie algebras $\gg_1$ and $\gg_2$.  

\subsection{Examples of nilpotent compatible Lie algebras}
We refer to the paper \cite{compatible Lie} for several examples of nilpotent compatible Lie algebras of small dimension. Beyond these, an important example follows from the theory of filiform Lie algebras. Let $\mathrm{L}_n$ be the Lie algebra with basis $\left\{e_0, e_1, \ldots, e_n\right\}$ defined by the products
$[e_0, e_i] = e_{i+1}$, where $1 \leq i \leq n-1$ (the other brackets being zero). In the literature, this algebra is usually refered to as the {\it model filiform Lie algebra}.
It is known that any filiform Lie algebra is isomorphic
to a linear deformation
of $\mathrm{L}_n$ \cite{Khakim}.

Indeed, let $\Delta$ be the set of pairs $(k, r) \in \mathbb{N}^2$ such that $1 \leq k \leq n - 1$, $2k + 1 < r \leq n$. If $n$ is odd, we suppose that $\Delta$ also contains  the pair $(\frac{n - 1}{2}, n)$. For any pair $(k, r) \in \Delta$, we can associate the 2-cocycle for the Chevalley-Eilenberg cohomology of $\mathrm{L}_n$ with coefficients in the adjoint module denoted $\Psi_{k, r}$ and defined by
$$\Psi_{k, r}(e_i, e_j) = -\Psi_{k, r}(e_j, e_i) = (-1)^{k-i}\binom{k-i}{j-k-1}
e_{i+j+r-2k-1},$$
when $1 \leq i \leq k < j \leq n$ and $\Psi_{k, r}(e_i, e_j) = 0$ otherwise. 
These formulas for $\Psi_{k, r}$ result from the conditions
$$\Psi_{k, r}(e_k, e_{k+1}) =e_r, \quad \Psi_{k, r} \in Z^2(\mathrm{L}_n, \mathrm{L}_n).$$

Let $\mathcal{F}_{n+1}$ be the variety of filiform Lie algebras of dimension $(n+1)$. Any filiform Lie algebra 
$\mu \in\mathcal{F}_{n+1}$ is isomorphic to $\mu_0 + \Psi$, where $\mu_0$ is the bracket of $\mathrm{L}_n$ and $\Psi$ is a $2$-cocycle given by
$\Psi = \sum_{(k, r) \in \Delta} a_{k, r}\Psi_{k, r},$
and verifying  
$\Psi \circ \Psi = 0$. In conclusion, any filiform Lie algebra gives rise to a  compatible Lie algebra $(\mathrm{L}_n, \mu_0, \Psi)$. Further examples are discussed in the upcoming sections.

\subsection{The nilradical of a compatible Lie algebra}

It is clear that the sum of ideals of a compatible Lie algebra is again an ideal. However, the sum of two nilpotent ideals of a compatible Lie algebra is not necessarily nilpotent. 

\begin{exa} Let $\cal L$ be $7$-dimensional compatible Lie algebra with basis $\left\{e_1, \ldots, e_7\right\}$ defined by
$$\begin{array}{cccccc}
\cal L:\left\{\begin{array}{lllll}\!\!
[e_1,e_2]=e_4,
[e_1,e_3]=e_5,
[e_2,e_3]=e_6,
[e_4,e_3]=e_7,
[e_1,e_6]=e_7,\\[2mm]\!\!
\{e_6,e_1\}=e_2,
\{e_1,e_7\}=e_4,
\{e_3,e_4\}=e_2,
\{e_7,e_3\}=e_6.
\end{array}\right.
\end{array}$$

The compatible Lie algebra $\cal L$ is non-nilpotent. Indeed, for the descending central series we have 
$$C^1(\cal L)=\textrm{span}\{e_2,e_4, e_5,e_6,e_7\}, 
\quad C^i(\cal L)=\textrm{span}\{e_2,e_4,e_6,e_7\}, \quad i\geq 2.$$
 
It can be verified that $\cal I_1=\textrm{span}\{e_1,e_2,e_4,e_5,e_6,e_7\}$ and $\cal I_2=\textrm{span}\{e_2,e_3,e_4,e_5,e_6,e_7\}$ are nilpotent ideals of the compatible Lie algebra $\cal L$, by a direct computation. However, $\cal L=\cal I_1+\cal I_2$ is not nilpotent neither. Moreover, $[[\cal I_2,\cal I_2]]$ is not an ideal of $\cal L$ because of $\{[[\cal I_2,\cal I_2]],\cal L\}=\textrm{span}\{e_2,e_4,e_6\}\not \subseteq[[\cal I_2,\cal I_2]].$ 
Thus, we can conclude that the sum of nilpotent ideals of a compatible Lie algebra might not be nilpotent and that the product of ideals might not be an ideal. 
\end{exa}

This example shows that the usual notion of nilradical is not well-defined for compatible Lie algebras. Therefore, we introduce an alternative definition. We begin with the notion of a special ideal, defined as follows.

\begin{defn} An ideal $\cal I$ of a compatible Lie algebra $\cal L$ is called a special ideal of $\cal L$, if for any ideal $\cal J$ of $\cal L$ the space $[[\cal I,\cal J]]$ is an ideal of $\cal L$.
\end{defn}

One can check that the sum of special ideals of a compatible Lie algebra is a special ideal. Similarly, their product is again special ideal. Moreover, for an arbitrary special ideal $\cal I$ of a compatible Lie algebra $\cal L$, we obtain that $C^i(\cal I)$ is an ideal of $\cal L$ for any $i\in \mathbb{N}.$ 
A special ideal  $\cal I$ is called a special nilpotent ideal of $\cal L$, if there exists $s\in\mathbb{N}$, such that $C^s(\cal I)=0$. Observe that since a compatible Lie algebra $\cal L$ is a special ideal of itself, then special nilpotency means just its nilpotency.

\begin{lem} \label{lem3.6} Sum of special nilpotent ideals of a compatible Lie algebra $\cal L$ is also a special nilpotent ideal of $\cal L$.
\end{lem}
\begin{proof} Let  $\cal I_1$ and $\cal I_2$ be special nilpotent ideals of a compatible Lie algebra $\cal L$. Then  $\cal I_1+\cal I_2,$ is a special ideal of $\cal L.$
Then there exists $s,t\in \mathbb{N}$, such that ${\cal C}^s(\cal I_1)=\{ 0 \}={\cal C}^t(\cal I_2)$. It is clear that ${\cal C}^{s+t}(\cal I_1+\cal I_2)$ is contained in a sum of terms $[[ [[\dots[[ [[\cal I_{\alpha_1},\cal I_{\alpha_2}]],\cal I_{\alpha_3}]],\dots]],\cal I_{\alpha_{s+t+1}}]],$ where $\alpha_i=\{1,2\}$. Taking into account that ${\cal C}^k(\cal I_i)$ is an ideal of $\cal L$ for any $k$, we derive that any such term is contained either in ${\cal C}^s(\cal I_1)$ or in ${\cal C}^t(\cal I_2)$. Therefore, we have ${\cal C}^{s+t}(\cal I_1+\cal I_2)=\{ 0 \},$ i.e., $\cal I_1+\cal I_2$ is a special nilpotent ideal of $\cal L$.
\end{proof}

By Lemma~\ref{lem3.6}, we obtain the existence of the maximal special nilpotent ideal called {\it special nilradical} or just nilradical. Thus, we can introduce the notion of {\em (special) solvable extension} of a nilpotent compatible Lie algebra $\gg$ as a solvable compatible Lie algebra $\gg'$ such that its special nilradical is $\gg$. Note that in the particular case of Lie algebras, these two notions coincide with the usual notions.

\section{Solvable extensions of compatible Lie algebras.}
\label{sec3}

The aim of this section is to explore the possibility of extending the Mubarakzyanov method to construct solvable extensions of nilpotent Lie algebras to the broader context of compatible Lie algebras, or alternatively to develop a method for constructing solvable extensions of a compatible Lie algebra (in some sense). However, this variety presents various difficulties which we illustrate in the following remarks.

\begin{rem}
    Consider the $n$-dimensional model filiform Lie algebra $\cal N_1$, given for a basis $\left\{e_1, \ldots, e_n\right\}$ by the products $\cal N_1 : [e_i,e_1]=e_{i+1}$ with $2\leq i \leq n-1.$
It is known that $\varphi_1, \varphi_2: \cal N_1 \times \cal N_1 \to \cal N_1$ given by 
$$\varphi_1(e_i,e_1)=e_{i+1} \textrm{ where }  2\leq i \leq n-1 \textrm{ and } \quad \varphi_2(e_1,e_2)=e_2$$
are $2$-cocycles of $\cal N_1$. Using these cocycles, we can construct  various types of compatible Lie algebras (nilpotent, solvable) with one part being $\cal N_1$. In particular, setting $ n=3$, the following $2$-cocycle 
 $$\varphi(e_2,e_1)=2e_1, \quad \varphi(e_2,e_3)=-2e_3, \quad \varphi(e_1,e_3)=e_2$$
gives us a {perfect compatible Lie algebra $(\cal N, [-,-],\varphi)$.} 
In particular for $n=4$, $\cal N_1$ is a $2$-cocycle of $sl_2 \oplus \mathbb{C}$. 
\end{rem}

By the theory of solvable Lie algebras, it is known that any nilpotent (non-characteristically nilpotent) Lie algebra $\cal N$ admits a solvable extension, that is, there exists a solvable (non-nilpotent) Lie algebra whose nilradical is precisely $\cal N$. In other words, the existence of this extensions depend on the existence of non-nilpotent derivations of a nilpotent Lie algebra \cite{Mubarakzyanov}. Since a compatible Lie algebra is a dialgebra, the derivation of a compatible Lie algebra is a derivation of both brackets. The following example shows that, even when all the derivations of a nilpotent compatible Lie algebra are nilpotent, a solvable extension may exist.

\begin{rem}\label{existcc} Let $\cal N$ be a nilpotent compatible Lie algebra with the following two component Lie algebras: 
$$\cal N_1:\left\{\begin{array}{llll}
[e_i,e_1]=e_{i+1},& 2\le i\le n-1,\\[3mm]
[e_i,e_2]=e_{i+2},& 3\le i\le n-2,
\end{array}\right. \quad \mbox{and} \quad \cal N_2:\ \left\{\begin{array}{llll}
\{e_2,e_1\}=e_{n-1},\\[3mm]
\{e_3,e_1\}=-e_{n}.
\end{array}\right.$$

In order to simplify the computational presentation, we assume $n=7$. It can be checked that any derivation of  $\textrm{Der}(\cal N)$ is nilpotent. However, we have a one-dimensional solvable extension: 
$$\mathcal{R}:\left\{\begin{array}{llllllllllllll}
 [e_i,e_1]=e_{i+1}, &2\le i\le 6, & \{e_2,e_1\}=e_6, &  \{e_1,e_3\}=e_7, &\\[3mm]
[e_i,e_2]=e_{i+2},&  3\le i\le 5,& \{e_1,x\}=a_1 e_6+a_2e_7,& \{e_2,x\}=a_3 e_6+a_4 e_7,&\\[3mm]
[e_i,x]=ie_i, &  1\le i\le 7,& \{e_3,x\}=-3e_6+a_3 e_7.&
\end{array}\right.$$
Indeed, it can be checked that $\mathcal{R}$  is a solvable compatible Lie algebra with special nilradical $\mathcal{N}$. 
\end{rem}

\medskip

Therefore, solvable extensions may exist without semisimple derivations. We next derive necessary conditions for a solvable extension $\mathcal N\rtimes \textrm{span}(z_1, \ldots, z_r)$ of a nilpotent compatible Lie algebra $\mathcal N$ with components $(\mathcal N_1,[\,,\,])$ and $(\mathcal N_2,\{\,,\,\})$.  In particular, such an extension must satisfy the following three conditions.

\begin{prop}
Let $(\cal N,[\, , \,],\{\, , \,\})$ be a compatible Lie algebra.
Then $\cal N\rtimes \textrm{span}(z_1, \ldots, z_r)$ is a compatible Lie algebra 
if and only if the following conditions hold for all $x, y \in \mathcal{N}$ and all $z\in \textrm{span}(z_1, \ldots, z_r)$.
\begin{enumerate}
\item $ad^{(1)}_z\in \textrm{Der}(\cal N_1)$.
\item $ad^{(2)}_z\in \textrm{Der}(\cal N_2)$.
\item For all $x,y\in\cal N$,
$$
ad^{(1)}_z\{x,y\}-\{ad^{(1)}_z x,\,y\}
-\{x,\,ad^{(1)}_z y\} +ad^{(2)}_z[ x,y ]
-[\,ad^{(2)}_z x,\,y]
-[\,x,\,ad^{(2)}_z y]
=0.
$$
\end{enumerate}
\end{prop}

The proof of the proposition is straightforward. Moreover, if $ad^{(1)}_z,\,ad^{(2)}_z\in \textrm{Der}(\cal N_1)\cap \textrm{Der}(\cal N_2)$, then the three conditions above hold. This is not the only possibility, as shown by the previous remark. However, these conditions alone do not ensure that $\cal N\rtimes \textrm{span}(z_1, \ldots, z_r)$ is a solvable extension with special nilradical $\cal N$. To identify the extra constraints, we study particular families of examples in the upcoming subsections.

\medskip

To set notation for the constructions below, we recall the next distinguished families of filiform algebras \cite{libroKluwer}:

\begin{enumerate}
    \item Let \( \mathrm{L}_n \) be the \(n\)-dimensional Lie algebra defined by the non-zero products
    $$[e_1,e_i] = e_{i+1}, \ 2\leq i \leq n-1,$$ 
    where $\{e_1, \dots, e_n\}$ is a basis of $\mathrm{L}_n$. This algebra is the model filiform Lie algebra.
    
    \item Let \( \mathrm R_n \) be defined in the basis $\{e_1, \dots, e_n\}$ by
    \[
    [e_1,e_i] = e_{i+1}, \quad 2\leq i \leq n-1, \quad [e_2,e_i] = e_{i+2}, \quad 3 \leq i \leq n-2.
    \]

\item Let \(\mathrm W_n \) be the Lie algebra whose brackets, in the basis $\{e_1, \dots, e_n\}$, are:

\begin{equation}\label{Wn}
\left\{\begin{array}{lllll}
[e_1, e_i] = e_{i+1}, & 2\le i\le n-1,\\[3mm]
[e_i, e_j] = \frac{6(j-i)(i-2)!(j-2)!}{(i+j-2)!}e_{i+j}, & 2 \leq i,j \leq n-2, \quad i+j \leq n.\\[3mm]
\end{array}\right.
\end{equation}

 It is isomorphic to the Lie algebra defined by:
\[[e_i, e_j] = (j-i) e_{i+j}, \quad 3 \leq i+j \leq n,\]
which is a finite quotient of the nilpotent part of the Witt algebra.
\end{enumerate}

\begin{rem}
    It is known that there are only two, up to isomorphism, $n$-dimensional Lie algebras, $n \geq 12$, whose brackets satisfy the relations:
\[
\begin{array}{llll}
[e_1,e_i]=e_{i+1}, & 2 \leq i \leq n-1,&\\[3mm]
[e_i, e_j] = a_{i,j} e_{i+j} & 2 \leq i < j \leq n-1, & i + j \leq n.
\end{array}\]
These Lie algebras are $\mathrm R_n$ and $\mathrm W_n$.

\end{rem}

\begin{prop}
    Any dialgebra $(\cal N,[\, , \,],\{\, , \,\})$ where $(\cal N,[\, , \,])$  is $\mathrm L_n$ and $(\cal N, \{\, , \,\})$ is an algebra from $\{\mathrm R_n, \mathrm W_n\}$ is a compatible nilpotent Lie algebra on the common basis $\{e_1,\dots,e_n\}$.
\end{prop}
\begin{proof}
Write $\mathrm R_n=\mathrm L_n+\Psi_R$ with $\Psi_R(e_2,e_i)=e_{i+2}$ ($3\le i\le n-2$), and $\mathrm W_n=\mathrm L_n+\Psi_W$ with $\Psi_W$ given by \eqref{Wn} on $i,j\ge 2$. Both are linear deformations of $\mathrm L_n$, hence
$d\Psi_R=d\Psi_W=0$ and $\Psi_R\circ\Psi_R=\Psi_W\circ\Psi_W=0$.
Therefore for all $\lambda_1, \lambda_2 \in \mathbb{C}$ the brackets $\mu_{\lambda_1,\lambda_2}=\lambda_1  \mathrm L_n+\lambda_2\Psi_R$ and $\mu'_{\lambda_1,\lambda_2}= \lambda_1 \mathrm L_n+\lambda_2\Psi_W$ define Lie algebras. Hence, the pairs $(\mathrm L_n,\mathrm R_n)$ or $(\mathrm L_n,\mathrm W_n)$ correspond to compatible algebras. Moreover, the nilpotency of these compatible algebras is clear.
\end{proof}

\begin{rem}
    The pair $(\mathrm R_n,\mathrm W_n)$ does not correspond to a compatible algebra because $\Psi_R\circ\Psi_W+\Psi_W\circ\Psi_R\neq 0$. 
\end{rem}

In the following we consider a nilpotent compatible Lie algebra 
with components from the set $\{\mathrm L_n, \mathrm R_n, \mathrm W_n\}$ and study its solvable extensions by means of non-nilpotent derivations. By choosing non-nilpotent derivations to construct the extension, it is clear that the nilradical will be $\mathcal{N}$, so the construction will be a solvable extension.

\subsection{Solvable extensions of the nilpotent compatible Lie algebra with Lie parts $\mathrm{L}_n$ and $\mathrm R_n$.}

First we consider the case of solvable compatible Lie extensions of a nilpotent compatible Lie algebra with component Lie algebras $\mathrm L_n$ and $\mathrm R_n$. Throughout this section, we denote by $\mathcal{N}$ the corresponding compatible Lie algebra. The derivations of the algebra $\mathcal{N}$ are presented in the following remark. 

\begin{rem}
    Let $\mathcal{N}$ be the compatible Lie algebra with basis $\left\{e_1, \ldots, e_n\right\}$ corresponding to the pair $(\mathrm L_n,\mathrm R_n)$. Then any derivation $d\in \textrm{Der}(\mathcal{N}) = \textrm{Der}(\mathrm{L}_n)\cap \textrm{Der}(\mathrm{R}_n)$ has the form 
    $$\left\{\begin{array}{lllll}
d(e_1)=\alpha_{1,1}e_1+\alpha_{1,n-1}e_{n-1}+\alpha_{1,n}e_n,\\[3mm]
d(e_i)=i\alpha_{1,1}e_i+\sum\limits_{t=3}^{n-i+2}\alpha_{2,t}e_{t+i-2},& 2\le i\le n.
\end{array}\right.$$
    for certain parameters $\alpha_{1,1},\alpha_{1,n-1},\alpha_{1,n},\alpha_{2,3},\ldots,\alpha_{2,n}\in\mathbb{C}$. The computation is straightforward. 
\end{rem}

Notice that $\cal N$ admits semisimple derivations. Hence we study solvable extensions of $\cal N$ obtained from non-nilpotent derivations, considering separately the cases in which the derivation lies in $\mathrm L_n$, in $\mathrm R_n$, or in $\cal N$ itself. We do not claim that these cases exhaust all solvable extensions of $\cal N$: there might be solvable extensions arising from nilpotent derivations on both algebras, see Remark~\ref{existcc}; however, several nontrivial examples can be constructed following this approach.

\subsubsection{Solvable extensions of $\cal N$ via non-nilpotent derivations of $\mathrm L_n$.}
We recall the algebra of derivations of the Lie algebras $\mathrm L_n$ and $\mathrm R_n$ in the following lemmas. 

\begin{rem}\label{rmderL}
For the model filiform Lie algebra $\mathrm L_n$ on the basis $\{e_1,\dots,e_n\}$, every derivation $d\in \textrm{Der}(\mathrm L_n)$ is 
\[
\left\{
\begin{array}{ll}
d(e_1)=\displaystyle\sum_{t=1}^n\alpha_{1,t}e_t,\\[2mm]
d(e_i)=\big((i-2)\alpha_{1,1}+\alpha_{2,2}\big)e_i+\displaystyle\sum_{t=3}^{\,n-i+2}\alpha_{2,t}\,e_{t+i-2},\quad 2\le i\le n,
\end{array}
\right.
\]
for parameters $\alpha_{1,1},\ldots,\alpha_{1,n},\alpha_{2,2},\ldots,\alpha_{2,n}\in\mathbb{C}$. The maximal torus is two-dimensional. Thus, the Lie algebra $\mathrm{L}_n$ admits $k$-dimensional solvable extensions up to $k\leq 2$.
\end{rem}

\begin{rem}
For the filiform Lie algebra $\mathrm R_n$ on the basis $\{e_1,\dots,e_n\}$, every derivation $d\in \textrm{Der}(\mathrm R_n)$ is 
\[
\left\{
\begin{array}{ll}
d(e_1)=\alpha_{1,1}e_1+\displaystyle\sum_{t=3}^{\,n}\alpha_{1,t}e_t,\\[2mm]
d(e_2)=2\alpha_{1,1}e_2+\displaystyle\sum_{t=3}^{\,n}\alpha_{2,t}e_t,\\[2mm]
d(e_i)=i\alpha_{1,1}e_i+\alpha_{2,3}e_{i+1}+\displaystyle\sum_{t=4}^{\,n-i+2}(\alpha_{2,t}-\alpha_{1,t-1})e_{t+i-2},\quad 3\le i\le n,
\end{array}
\right.
\]
for parameters $\alpha_{1,1},\alpha_{1,3},\ldots,\alpha_{1,n},\alpha_{2,3},\ldots,\alpha_{2,n}\in\mathbb{C}$.
\end{rem}

Firstly, we describe the one-dimensional solvable extensions  by means of non-nilpotent derivations of $\mathrm{L}_n$.

\begin{thm}\label{thm6.2}
Let $\cal R$ be a one-dimensional solvable extension of $\cal N$ obtained from non-nilpotent derivations of $\mathrm L_n$. Then there is a basis $\{e_1,\dots,e_n,x\}$ such that the multiplication by $x$ is
\[
\left\{
\begin{array}{ll}
[e_i,x]=ie_i, & 1\le i\le n,\\[2mm]
\{e_1,x\}=\beta_{1,1}e_1+\beta_{1,n-1}e_{n-1}+\beta_{1,n}e_n,\\[2mm]
\{e_i,x\}=i\beta_{1,1}e_i+\displaystyle\sum_{t=3}^{\,n-i+2}\beta_{2,t}\,e_{t+i-2}, & 2\le i\le n,
\end{array}
\right.
\]
for parameters $\beta_{1,1},\beta_{1,n-1},\beta_{1,n},\beta_{2,3},\ldots,\beta_{2,n}\in\mathbb{C}$.
\end{thm}

\begin{proof}
From the derivations of $\mathrm L_n$ we may write
\[
\left\{
\begin{array}{ll}
[e_1,x]=\alpha_{1,1}e_1+\displaystyle\sum_{t=2}^n\alpha_{1,t}e_t,\\[2mm]
[e_i,x]=\big((i-2)\alpha_{1,1}+\alpha_{2,2}\big)e_i+\displaystyle\sum_{t=3}^{\,n-i+2}\alpha_{2,t}\,e_{t+i-2},\quad 2\le i\le n.
\end{array}
\right.
\]
Replacing $x$ by $x+\alpha_{2,3}e_1-\sum_{t=3}^n\alpha_{1,t}\,e_{t-1}$ gives
\[
[e_1,x]=\alpha_{1,1}e_1+\alpha_{1,2}e_2,\qquad
[e_i,x]=\big((i-2)\alpha_{1,1}+\alpha_{2,2}\big)e_i+\sum_{t=4}^{n-i+2}\alpha_{2,t}\,e_{t+i-2}.
\]
Using the description of $\textrm{Der}(\cal N_2)$ we obtain
\[
\left\{
\begin{array}{ll}
\{e_1,x\}=\beta_{1,1}e_1+\displaystyle\sum_{t=3}^{\,n}\beta_{1,t}e_t,\\[2mm]
\{e_2,x\}=2\beta_{1,1}e_2+\displaystyle\sum_{t=3}^{\,n}\beta_{2,t}e_t,\\[2mm]
\{e_i,x\}=i\beta_{1,1}e_i+\beta_{2,3}e_{i+1}+\displaystyle\sum_{t=4}^{\,n-i+2}(\beta_{2,t}-\beta_{1,t-1})e_{t+i-2},\quad 3\le i\le n.
\end{array}
\right.
\]
By the compatibility we obtain the constraints
\[
\alpha_{1,2}=0,\qquad \alpha_{2,2}=2\alpha_{1,1},\qquad \beta_{1,t}=0\ (3\le t\le n-2).
\]
Since the extension is non-nilpotent, then $\alpha_{1,1}\ne0$. We rescale $x$ so that $\alpha_{1,1}=1$.  
Now the change of basis
\[
e_1'=e_1,\qquad e_i'=e_i+\sum_{t=4}^{\,n-i+2}A_t\,e_{t+i-2}\quad (2\le i\le n),
\]
with coefficients $A_t$ chosen recursively, we obtain the stated form of the extension $\cal R$. Note that the fact that $\cal R$ has special nilradical $\cal N$ is clear.
%$$A_4=-\frac{\alpha_{2,4}}{2},\quad A_5=-\frac{\alpha_{2,5}}{3},\quad A_i=-\frac{1}{i-2}(\alpha_{2,i}+\sum\limits_{j=4}^{i-2}A_j\alpha_{2, i-j+2}),\quad 6\le i\le n,$$
\end{proof}

Next, we consider the description of the two-dimensional solvable extensions of $\cal N$ by means of non-nilpotent derivations of the Lie algebra $\mathrm{L}_n$.

 {
 
\begin{thm}\label{N_1(2)} There is no $(n+2)$-dimensional solvable compatible Lie extension of $\cal N$ by means of non-nilpotent derivations of $\mathrm{L}_n$. 
\end{thm}
\begin{proof} Let $\cal R$ be a $(n+2)$-dimensional solvable Lie algebras with nilradical $\mathrm{L}_n$ and $\{e_1,\dots, e_n, x,y\}$ is a basis of $\cal R$. We can write the products of $\cal R$ as
$$\left\{\begin{array}{lllllll}
[e_1,x]=\alpha_{1,1}e_1+\sum\limits_{t=2}^n\alpha_{1,t}e_t,\\[3mm]
[e_i,x]=((i-2)\alpha_{1,1}+\alpha_{2,2})e_i+\sum\limits_{t=i+1}^{n}\alpha_{2,t-i+2}e_t,& 2\le i\le n,\\[3mm]
[e_1,y]=\beta_{1,1}e_1+\sum\limits_{t=2}^n\beta_{1,t}e_t,\\[3mm]
[e_i,y]=((i-2)\beta_{1,1}+\beta_{2,2})e_i+\sum\limits_{t=i+1}^{n}\beta_{2,t-i+2}e_t,& 2\le i\le n,\\[3mm]
\end{array}\right.$$
where  $\alpha_{1,1}\beta_{2,2}-\alpha_{2,2}\beta_{1,1}\neq 0$. Replacing $x$ and $y$ by   
$$x'=\frac{1}{\big(\alpha_{1,1}\beta_{2,2}-\alpha_{2,2}\beta_{1,1}\big)}\big(\beta_{2,2}x-
\alpha_{2,2}y\big),\quad 
y'=\frac{1}{\big(\alpha_{1,1}\beta_{2,2}-\alpha_{2,2}\beta_{1,1}\big)}\Big(\alpha_{1,1}y-\beta_{1,1}x\Big),$$ 
we can assume the products is:
$$\left\{\begin{array}{lllll}
[e_1,x]=e_1+\sum\limits_{t=2}^n\alpha_{1,t}e_t,\\[3mm]
[e_i,x]=(i-2)e_i+\sum\limits_{t=i+1}^{n}\alpha_{2,t-i+2}e_t,& 2\le i\le n,\\[3mm]
[e_1,y]=\sum\limits_{t=2}^n\beta_{1,t}e_t,\\[3mm]
[e_i,y]=e_i+\sum\limits_{t=i+1}^{n}\beta_{2,t-i+2}e_t,& 2\le i\le n.
\end{array}\right.$$
In addition, setting $x'=x+\alpha_{2,3}e_1-\sum\limits_{t=2}^{n-1}\alpha_{1,t+1}e_{t}$ and $y'=y+\beta_{2,3}e_1-\sum\limits_{t=2}^{n-1}\beta_{1,t+1}e_t,$
we can assume 
$$\begin{array}{llll}
[e_1,x]=e_1+\alpha_{1,2}e_2, & [e_i,x]=(i-2)e_i+\sum\limits_{t=i+2}^{n}\alpha_{2,t-i+2}e_t, & 2\le i\le n,\\[3mm]
[e_1,y]=\beta_{1,2}e_2, & [e_i,y]=e_i+\sum\limits_{t=i+2}^{n}\beta_{2,t-i+2}e_t, & 2\le i\le n. \\[3mm]
\end{array}$$

Next, by choosing the new basis:  $$e_1'=e_1,\quad %e_2'=e_2+\sum\limits_{t=4}^{n}A_te_t,\quad
 e_i'=e_i+\sum\limits_{t=4}^{n-i+2}A_te_{t+i-2},\ \  2\le i\le n, \quad \mbox{with} \quad A_t=\frac{1}{2-t}\big(\alpha_{2,t}+\sum\limits_{p=4}^{t-2}\alpha_{2,p}A_{t-p+2}),$$
we obtain $[e_i,x]=(i-2)e_i, \ 2\le i\le n.$ 
It should be noted that under the above basis transformation the multiplications table of $\mathrm{R}_n$ does not change.
Applying the Jacobi identity for the obtained products and using the change $x'=x-\gamma e_n$ we obtain 
$$\cal R:\left\{\begin{array}{lllll}
[e_1,x]=e_1+\alpha e_2,& [e_1,y]=-\alpha e_2,\\[3mm]
[e_i,x]=(i-2)e_i,& [e_i,y]=e_i,& 2\le i\le n.
\end{array}\right.$$
For $\varphi\in Hom(\cal R \wedge \cal R, \cal N)$ such that $\{e_i,e_j\}=\varphi(e_i,e_j)$ and
$\varphi(e_i,y)=\sum\limits_{t=1}^{n}\alpha_{i,t}e_t$, we verify the compatibility.

\begin{center}
    \begin{tabular}{lll}
       \qquad {2-cocyle identity} & &\quad\quad\quad Constraints\\
        \hline \hline
\\
        $Z(e_1,e_i,y)=0,\quad 2\le i\le 5,$ &\quad $\Rightarrow $\quad &\quad $\left\{\begin{array}{lll}
        \alpha_{i,t}=0,&         \alpha_{i,i}=(i-2)\alpha_{1,1}+\alpha_{2,2},\\[3mm]
        1\le t<i,&3\le i\le 5,\end{array}\right.$ \\[6mm]
        
        $Z(e_2,e_3,y)=0,$ &\quad $\Rightarrow $\quad &\quad $\alpha_{2,2}=2\alpha_{1,1}$,\quad $\alpha_{2,1}=0.$\\[2mm]
        \end{tabular}
\end{center}

Finally, we get 
$Z(e_2,e_4,y)=
%[e_2,\varphi(e_4,y)]-[\varphi(e_2,e_4),y]+[\varphi(e_2,y),e_4]+\varphi(e_2,[e_4,y])-\varphi([e_2,e_4],y)+\varphi([e_2,y],e_4)=$$$$
%=4B_{1,1}e_6-e_6+2B_{1,1}e_6+e_6-%6B_{1,1}e_6+e_6+\sum\limits_{t=7}^n(*)e_t=
e_6+\sum\limits_{t=7}^n(*)e_t\neq0.$ This contradiction completes the proof of theorem. 
%Therefore $\cal R_2$ does not have 2-Lie cocycle suitable for a Lie algebra which has ideal $\mathrm{R}_n$.
\end{proof}

}

\subsubsection{Solvable extensions of $\cal N$ via non-nilpotent derivations of $\mathrm R_n$.}

We now describe the one-dimensional solvable extensions arising from non-nilpotent derivations of $\mathrm R_n$. Since the maximal torus in $\textrm{Der}(\mathrm R_n)$ is one-dimensional, only one-dimensional solvable extensions can occur.

\begin{thm}\label{thm6.4}
Let $\cal R$ be a one-dimensional solvable extension of $\cal N$ obtained from non-nilpotent derivations of $\mathrm R_n$. Then there exists a basis $\{e_1,\dots,e_n,x\}$ such that the multiplication by $x$ is
\[
\left\{
\begin{array}{ll}
[e_1,x]=\alpha_{1,1}e_1+\alpha_{1,n-1}e_{n-1}+\alpha_{1,n}e_n,\\[2mm]
[e_i,x]=i\,\alpha_{1,1}e_i+\displaystyle\sum_{t=4}^{\,n-i+2}\alpha_{2,t}\,e_{t+i-2}, & 2\le i\le n,\\[2mm]
\{e_i,x\}=i\,e_i, & 1\le i\le n,
\end{array}
\right.
\]
for parameters $\alpha_{1,1},\alpha_{1,n-1},\alpha_{1,n},\alpha_{2,4},\ldots,\alpha_{2,n}\in\mathbb{C}$.
\end{thm}
\begin{proof} By the description of $\textrm{Der}(\mathrm{R}_n)$, we get the following products:
\begin{equation}\label{eq11}
\left\{\begin{array}{llllllll}
\{e_1,x\}=\beta_{1,1}e_1+\sum\limits_{t=3}^{n}\beta_{1,t}e_t,\\[3mm]
\{e_2,x\}=2\beta_{1,1}e_2+\sum\limits_{t=3}^{n}\beta_{2,t}e_t,\\[3mm]
\{e_i,x\}=i\beta_{1,1}e_i+\beta_{2,3}e_{i+1}+\sum\limits_{t=4}^{n-i+2}(\beta_{2,t}-\beta_{1,t-1})e_{t+i-2},& 3\le i\le n.
\end{array}\right.
\end{equation}

Putting $x'=\frac{1}{\beta_{1,1}}(x+{\beta_{2,3}}e_1-\sum\limits_{t=3}^{n}\beta_{1,t}e_{t-1}),$
we can assume that \eqref{eq11} has the form
$$\left\{\begin{array}{lllllll}
\{e_1,x\}=e_1,\\[3mm]
%\{e_2,x\}=2e_2+\sum\limits_{t=4}^{n}\beta_{2,t}e_t,\\[3mm]
\{e_i,x\}=ie_i+\sum\limits_{t=4}^{n-i+2}\gamma_te_{t+i-2},& 2\le i\le n.
\end{array}\right.$$
%(where $\gamma_t=\beta_{2,t}-\beta_{1,t-1}.$)

Applying a triangular change of basis of the form
\[
e_1'=e_1,\qquad e_i'=e_i+\sum_{t=4}^{\,n-i+2}A_t\,e_{t+i-2}\quad(2\le i\le n),
\]
with coefficients \(A_t\) chosen recursively so as to cancel the higher terms, we obtain $\{e_i,x\}=i\,e_i\qquad(1\le i\le n).$
This change preserves the multiplication tables of \(\mathrm L_n\) and \(\mathrm R_n\), so we obtain the stated normal form of the extension \(\cal R\) with special nilradical \(\cal N\).
Moreover, by the description of $\textrm{Der}(\mathrm{L}_n)$, the products of the solvable extension of $\mathrm{L}_n$ are the followings:
$$\left\{\begin{array}{lllll}
[e_1,x]=\alpha_{1,1}e_1+\sum\limits_{t=2}^n\alpha_{1,t}e_t,\\[3mm]
[e_i,x]=((i-2)\alpha_{1,1}+\alpha_{2,2})e_i+\sum\limits_{t=i+1}^{n}\alpha_{2,t-i+2}e_t,& 2\le i\le n.\\[3mm]
\end{array}\right.$$

Finally, the compatibility condition yields the constraints:
\begin{center}
    \begin{tabular}{llllll}
        {Compatibility Rule} & & \quad &\quad  Constraints\\
        \hline \hline
\\
        $L(e_2,e_3,x)=0,$ &\quad &\quad $\Rightarrow $\quad\quad & $\alpha_{2,2}=2\alpha_{1,1},$\\[3mm]
        
        $L(e_1,e_2,x)=0,$ &\quad &\quad $\Rightarrow $\quad\quad & $\alpha_{1,t}=0,\quad 2\le t\le n-2.$
        \end{tabular}
\end{center}
The stated multiplication then follows. \end{proof}

\subsubsection{Solvable extensions of $\cal N$ via non-nilpotent derivations of $\cal N$.}

We study one-dimensional solvable extensions of $\cal N$ obtained from non-nilpotent derivations of $\cal N$ itself. Since the maximal torus of $\textrm{Der}(\cal N)$ is one-dimensional, only one-dimensional solvable extensions can arise in this case.

\begin{thm} Let $\cal R$ be a one-dimensional solvable extension of $\cal N$ by means of non-nilpotent derivations of $\textrm{Der}(\mathrm{L}_n)\cap \textrm{Der}(\mathrm{R}_n)$. Then there exists a basis $\{e_1,\dots,e_n,x\}$ such that the multiplication by $x$ is
$$\left\{\begin{array}{llllll}
[e_i,x]=ie_i,&1\le i\le n,\\[3mm]
\{e_1,x\}=\beta_{1,1}e_1+\beta_{1,n-1}e_{n-1}+\beta_{1,n}e_n,\\[3mm]
\{e_i,x\}=i\beta_{1,1}e_i+\sum\limits_{t=3}^{n-i+2}\beta_{2,t}e_{t+i-2},&2\le i\le n.
\end{array}\right.$$
for parameters $\beta_{1,1},\beta_{1,n-1},\beta_{1,n},\beta_{2,3},\ldots,\beta_{2,n}\in\mathbb{C}$.
\end{thm}
\begin{proof} By the derivations of $\textrm{Der}(\mathrm{L}_n)\cap \textrm{Der}(\mathrm{R}_n)$ we have the products in $\cal R$:
$$\left\{\begin{array}{lllll}
[e_1,x]=\alpha_{1,1}e_1+\alpha_{1,n-1}e_{n-1}+\alpha_{1,n}e_n,&\{e_1,x\}=\beta_{1,1}e_1+\beta_{1,n-1}e_{n-1}+\beta_{1,n}e_n,\\[3mm]
[e_i,x]=i\alpha_{1,1}e_i+\sum\limits_{t=3}^{n-i+2}\alpha_{2,t}e_{t+i-2},&
\{e_i,x\}=i\beta_{1,1}e_i+\sum\limits_{t=3}^{n-i+2}\beta_{2,t}e_{t+i-2},&2\le i\le n.\\[3mm]
\end{array}\right.$$

Clearly, the solvability of  $\cal R$ implies that  $(\alpha_{1,1},\beta_{1,1})\neq(0,0)$. 
Suppose $\alpha_{1,1}\neq0$, then by setting 
$$x'=\frac{1}{\alpha_{1,1}}(x+\alpha_{1,n-1}e_{n-2}+\alpha_{1,n}e_{n-1}),$$
%($x'=\frac{x+\beta_{1,n-1}e_{n-2}+\beta_{1,n}e_{n-1}}{\beta_{1,1}}$) then  without loss of generalities 
we obtain the multiplication
$$[e_1,x]=e_1,\quad [e_i,x]=ie_i+\sum\limits_{t=3}^{n-i+2}\alpha_{2,t}e_{t+i-2},\quad 2\le i\le n.$$
Moreover, by choosing the new basis  
$$e_1'=e_1,\quad e_i'=e_i+\sum\limits_{t=4}^{n-i+2}A_te_{t+i-2},\quad 2\le i\le n,$$ with the coefficients 
$$A_4=-\frac{\alpha_{2,4}}{2},\quad \quad A_i=-\frac{1}{i-2}(\alpha_{2,i}+\sum\limits_{j=3}^{i-2}A_j\alpha_{2,i-j+2}),\quad 5\le i\le n,$$ we obtain $[e_i,x]=ie_i$, for $1\le i\le n.$
Similarly, when $\beta_{1,1}\neq 0$, the same normalization is achieved and the two situations are completely symmetric, so we present them jointly in the statement.
\end{proof}

\subsection{Solvable extensions of the nilpotent compatible Lie algebra with Lie parts $\mathrm{L}_n$ and $\mathrm W_n$.}

Now, we consider the case of solvable compatible Lie extensions of a nilpotent compatible Lie algebra with component Lie algebras $\mathrm L_n$ and $\mathrm W_n$. Throughout this section, we denote by $\mathcal{N}$ the corresponding compatible Lie algebra. Recall the following result about the outer derivations of .

\begin{rem}
A basis of $H^1(\mathrm{W}_n, \mathrm{W}_n)$ is given by  
the following maps (see \cite{libroKluwer})
$$\left\{\begin{array}{llll}
h(e_i) = ie_i & 1\leq i \leq n,\\[3mm]
t_1(e_ 1)=e_n & t_1(e_i)=0,& i \neq 1,\\[3mm]
t_2(e_i) =e_{n-4+i} & 2\le i \le 4, & t_2(e_i)=0,&i \geq 5,\\[3mm]
t_3(e_i)=e_{n-3+i} &  2\le i\le 3, & t_3(e_i)=0,&i \geq 4,
\end{array}\right.$$
\end{rem}

The algebra $\cal N$ admits semisimple derivations. In particular, the maximal torus is one-dimensional.  Hence we study solvable extensions of $\cal N$ obtained from non-nilpotent derivations, considering separately the cases in which the derivation lies in $\mathrm L_n$, in $\mathrm W_n$ or in $\cal N$.

\subsubsection{Solvable extensions of $\cal N$ via non-nilpotent derivations of $\mathrm L_n$.}
We begin describing the one-dimensional solvable extensions  by means of non-nilpotent derivations of $\mathrm{L}_n$.

{\begin{thm}\label{thm4.12} Let $\cal R$ be a one-dimensional solvable extension of $\cal N$ by means of non-nilpotent derivations of $\mathrm{L}_n$. Then there exists a basis $\{e_1,\dots,e_n,x\}$ such that the multiplication by $x$ is $$\left\{\begin{array}{lllllll} [e_1,x]=e_1,\\[3mm] [e_i,x]=ie_i+\sum\limits_{k={{n-2}}}^{n}\alpha_{2,k}e_{k+i-2},& 2\le i\le n,\\[3mm] \{e_1,x\}=\gamma e_1+\gamma_{n-2} e_{n-1}+(\gamma_{n-1}+\tau_1)e_n,&\\[3mm] \{e_i,x\}=i\gamma e_i+\gamma_1\{e_i,e_1\}+\gamma_{n-2}\{e_i,e_{n-2}\}+\sum\limits_{k=2}^{3}\tau_kt_k(e_i),& 2\le i\le n.& \end{array}\right.$$ 
for parameters 
$\gamma,\gamma_1,\gamma_{n-2},\gamma_{n-1},\tau_1,\tau_2,\tau_3,\alpha_{2,4},\ldots,\alpha_{2,n}\in\mathbb{C}$.  
\end{thm}

}

{ 
\begin{proof} By using the basis transformation in the proof of Theorem \ref{thm6.2}, we can write the product as
$$\left\{\begin{array}{lllll}
[e_1,x]=\alpha_{1,1}e_1+\alpha_{1,2}e_2,\\[3mm]
[e_i,x]=((i-2)\alpha_{1,1}+\alpha_{2,2})e_i+\sum\limits_{k=i+2}^{n}\alpha_{2,k-i+2}e_k,& 2\le i\le n.\\[3mm]
\{e_1,x\}=\gamma e_1+\sum\limits_{k=3}^{n-1}\gamma_{k-1} e_{k}+(\gamma_{n-1}+\tau_1)e_n,&\\[3mm]
\{e_i,x\}=i\gamma e_i+\sum\limits_{k=1}^{n-i}\gamma_k\{e_i,e_k\}+\sum\limits_{k=2}^{3}\tau_kt_k(e_i),& 2\le i\le n.&
\end{array}\right.$$

Then, the compatibility condition give us the following constraints.
\begin{center}
    \begin{tabular}{llllll}
        {Compatibility Rule} & & \quad &\quad  Constraints\\
        \hline \hline
\\ $L(e_1,e_2,x)=0,$ &\quad &\quad $\Rightarrow $\quad\quad & $ \gamma_k=0,\quad 2\le t\le n-3,$\\[3mm]      $L(e_2,e_3,x)=0,$ &\quad  &\quad$\Rightarrow $\quad\quad & $ {\alpha_{2,k}=0,\quad 4\le n-3,\quad \alpha_{2,2}=2\alpha_{1,1},}$\\[3mm] 
$L(e_1,e_3,x)=0,$ &\quad &\quad $\Rightarrow $\quad\quad & $\alpha_{1,2}=0.$ \\[3mm]        
    \end{tabular}
\end{center}
Obviously, we have $\alpha_{1,1}\neq0$ to ensure the non-nilpotency of the extension. Therefore, setting  $x'=\frac{x}{\alpha_{1,1}}$, we can assume $\alpha_{1,1}=1$ and the multiplication table of the statement is obtained. 
\end{proof}

Now, consider two-dimensional solvable extensions of $\mathrm{L}_n$ by means of non-nilpotent derivations of $\mathrm{L}_n$. Then applying derivation properties along with basis transformation used in the proof of Theorem \ref{N_1(2)} we obtain the products of the solvable algebra $\cal R$. The compatibility on the triple of elements $\{e_2,e_3,x_2\}$ gives us the condition   
$L(e_2,e_3,x_2)=-2e_5+\sum_{t=6}^n(*)e_t\neq0$. This contradiction implies the following result.
\begin{thm}\label{thm4.10} There is no $(n+2)$-dimensional solvable compatible Lie extension of $\cal N$ by means of non-nilpotent derivations of $\mathrm{L}_n$. 
\end{thm}
}

\subsubsection{Solvable extensions of $\cal N$ via non-nilpotent derivations of $\mathrm W_n$.}

Now, we describe the one-dimensional solvable extensions arising from non-nilpotent derivations of $\mathrm W_n$. Since the maximal torus in $\textrm{Der}(\mathrm W_n)$ is one-dimensional, only one-dimensional solvable extensions appear.

\begin{thm}\label{thm4.13}
Let $\cal R$ be a one-dimensional solvable extension of $\cal N$ by means of non-nilpotent derivations of $\mathrm W_n$. Then there exists a basis $\{e_1,\dots,e_n,x\}$ such that the multiplication by $x$ is
\[
\left\{
\begin{array}{ll}
[e_1,x]=\alpha_{1,1} e_1+\alpha_{1,n-1} e_{n-1}+\alpha_{1,n}e_n,\\[2mm]
%[e_i,x]=i\,\alpha_{1,1}e_i+\displaystyle\sum_{k=i+1}^{n+i-2}\alpha_{2,k-i+2}\,e_k, & 2\le i\le n,\\[2mm]
{[e_i,x]=i\,\alpha_{1,1}e_i+\displaystyle\sum_{k=n-2}^{n}\alpha_{2,k}\,e_{k+i-2},} & 2\le i\le n,\\[2mm]

\{e_1,x\}=e_1,\qquad \{e_2,x\}=2e_2+\tau_2 e_{n-2}+\tau_3 e_{n-1},\\[2mm]
\{e_3,x\}=3e_3+\tau_2 e_{n-1}+\tau_3 e_n,\qquad \{e_4,x\}=4e_4+\tau_2 e_n,\\[2mm]
\{e_i,x\}=i e_i, & 5\le i\le n,
\end{array}
\right.
\]
for parameters $\alpha_{1,1},\alpha_{1,n-1},\alpha_{1,n},\alpha_{2,3},\ldots,\alpha_{2,n},\tau_2,\tau_3\in\mathbb{C}$. 
\end{thm}
\begin{proof} Using the descripcion of the derivations, we obtain the products of $\cal R$.
$$\left\{\begin{array}{llll}
[e_1,x ]=\alpha_{1,1} e_1+\sum\limits_{t=2}^n\alpha_{1,t} e_t,& \{e_1,x\}=e_1,\\[3mm]
[e_i,x ]=((i-2)\alpha_{1,1} +\alpha_{2,2} )e_i+\sum\limits_{k=i+1}^{n}\alpha_{2,k-i+2} e_k,&  2\le i\le n,\\[3mm]
\{e_2,x\}=2e_2+\tau_2e_{n-2}+\tau_3e_{n-1},&\{e_4,x\}=4e_4+\tau_2e_n, \\[3mm]
\{e_3,x\}=3e_3+\tau_2e_{n-1}+\tau_3e_n,&\{e_i,x\}=ie_i,\quad 5\le i\le n.\end{array}\right.$$

Then the result follows by finding the constraints obtained from the compatibility rule. 
\begin{center}
    \begin{tabular}{llllll}
        {Compatibility Rule} & & \quad &\quad  Constraints\\
        \hline \hline
\\
        $L(e_1,e_2,x)=0,$ &\quad &\quad $\Rightarrow $\quad\quad & $\alpha_{1,t}=0,\quad 2\le t\le n-2,$\\[3mm]
        $L(e_2,e_3,x)=0,$ &\quad  &\quad$\Rightarrow $\quad\quad & ${\alpha_{2,t}=0,\quad 3\le t\le n-3},\quad \alpha_{2,2}=2\alpha_{1,1}.$\\[3mm] 
    \end{tabular}
\end{center}
\end{proof}

\subsubsection{Solvable extensions of $\cal N$ via non-nilpotent derivations of $\cal N$.}

Lastly, we study one-dimensional solvable extensions of $\cal N$ obtained from non-nilpotent derivations of $\cal N$ itself. 

\begin{thm} Let $\cal R=\cal N\oplus \textrm{span}\{x\}$ be a one-dimensional solvable extension of $\cal N$ by means of non-nilpotent derivations of $\textrm{Der}(\mathrm{L}_n)\cap \textrm{Der}(\mathrm{W}_n)$. Then in addition to the products of $\cal N$ the algebra $\cal R$ admits one of the following multiplications table
%two Lie components $(\cal R_1, [-,-])$ and $(\cal R_2, \{-,-\})$ with 
$$\left\{\begin{array}{lllllll}
[e_1,x]= e_1,&
\{e_1,x\}=\gamma e_1+\gamma_1e_{n-1}+\gamma_2 e_n,\\[3mm]
[e_2,x]=2 e_2+\alpha_3e_{n-2}+\alpha_4e_{n-1},& 
\{e_2,x\}=2\gamma e_2+\gamma_3e_{n-2}+\gamma_4e_{n-1}+\frac{(n-4)\gamma_1}{(n-3)(n-2)}e_{n},\\[3mm]
[e_3,x]=3e_3+\alpha_3e_{n-1}+\alpha_4e_{n},& 
\{e_3,x\}=3\gamma e_3+\gamma_3e_{n-1}+\gamma_4e_{n},\\[3mm]
[e_4,x]=4 e_4+\alpha_3e_{n},& 
\{e_4,x\}=4\gamma e_4+\gamma_3e_n,&\\[3mm]
[e_i,x]=i e_i,&
\{e_i,x\}=i\gamma e_i,\quad 5\le i\le n,
\end{array}\right.$$
for $\alpha_{3},\alpha_{4},\gamma,\gamma_1,\gamma_2,\gamma_3,\gamma_4\in\mathbb{C}$, and where the products $[e_i,x]$ 
can be interchanged with the products $\{e_i,x\}$. 
\end{thm}

\begin{proof} The description of $\textrm{Der}(\mathrm{L}_n)\cap \textrm{Der}(\mathrm{W}_n)$ leads to the following products for $\mathcal{R}$
$$\left\{\begin{array}{lllllll}
[e_1,x]=\alpha e_1+\alpha_{1}e_{n-1}+\alpha_2e_n,
&\{e_1,x\}=\gamma e_1+\gamma_1e_{n-1}+\gamma_2 e_n,\\[3mm]
[e_2,x]=2\alpha e_i+\alpha_3e_{n-2}+\alpha_4e_{n-1}+\frac{(n-4)\alpha_1}{(n-3)(n-2)}e_{n},&\{e_2,x\}=2\gamma e_2+\gamma_3e_{n-2}+\gamma_4e_{n-1}+\frac{(n-4)\gamma_1}{(n-3)(n-2)}e_n, \\[3mm]
[e_3,x]=3\alpha e_3+\alpha_3e_{n-1}+\alpha_4e_{n},&\{e_3,x\}=3\gamma e_3+\gamma_3e_{n-1}+\gamma_4e_{n},\\[3mm]
[e_4,x]=4\alpha e_4+\alpha_3e_{n},& \{e_4,x\}=4\gamma e_4+\gamma_3e_n,\\[3mm]
[e_i,x]=i\alpha e_i,&\{e_i,x\}=i\gamma e_i,\quad  5\le i\le n.
\end{array}\right.$$

It can be checked that any such product is compatible. The solvability of $\cal R$ implies that $(\alpha,\gamma)\neq(0,0)$.  Then there are two possibilities. On the one hand, if $\alpha\neq0$, then we can choose
$x'=\frac{1}{\alpha}\big(x-\alpha_{1}e_{n-2}-\alpha_{2}e_{n-1}\big)$ and the multiplication table follows accordingly. On the other hand, if $\gamma\neq0$, then setting 
$x'=\frac{1}{\gamma}\big(x-\gamma_1e_{n-2}-\gamma_2 e_{n-1}\big)$ leads to the corresponding multiplication table.
\end{proof}

\section{Filiform compatible Lie algebras}

Among nilpotent algebras, the filiform class is arguably the most significant, characterized by having the maximal possible nilpotency index. It is evident that the nulfiliform case does not produce any non-abelian compatible Lie algebra, since any generating set of a compatible Lie algebra must contain at least two elements. Consequently, we focus our attention directly on the filiform case.

Denote by $P$ the adjoint operator of $(\mathcal{L}, [-, -])$ and by $Q$ the adjoint operator of $(\mathcal{L}, \left\{-, -\right\})$. Throughout, we assume the base field is $\mathbb{C}$, although, some arguments are also valid over any infinite field.

\begin{defn}
    Let $\mathcal{L}$ be a nilpotent compatible Lie algebra of dimension $n+1$. We say $\mathcal{L}$ is filiform if we have $\textrm{dim}\,  \mathcal{L}^i = n+1-i$ for any $1 \leq i\leq n+1$.
\end{defn}

Let $\mathcal{L}$ be a filiform compatible Lie algebra. In the next result, we construct a basis of $\mathcal{L}$ that satisfies a set of desirable properties. This basis is a reminiscent of the adapted basis used in the study of filiform Lie algebras.

\begin{lem}\label{lemmafili1}
    Let $\mathcal{L}$ be a filiform compatible Lie algebra of dimension $n+1$. 
    Then there is a basis $\left\{e_0, \ldots, e_{n}\right\}$ of $\mathcal{L}$  and natural numbers $(n_1, \ldots, n_t)$, with  $\underline{n_{j}} = \sum_{i=1}^{j} n_i$ and $n= \underline{n_t}$, such that 
        \begin{align*}
            &[e_0, e_i] = e_{i+1}, & \textrm{for $1\leq i\leq \underline{n_1}$,} \\
            &[M, e_i] = 0, \quad \quad \quad \quad  \quad \quad \left\{e_0, e_i\right\} = e_{i+1}, & \textrm{for $\underline{n_1} < i\leq \underline{n_2}$,} \\
            &[e_0, e_i] = e_{i+1}, & \textrm{for $\underline{n_2}< i\leq \underline{n_3}$,}\\
            & \ldots & \ldots\\
            &[e_0, e_i] = e_{i+1}, & \textrm{for $\underline{n_{t-1}} < i\leq \underline{n_t}$ (if $t$ is odd),} \\
            &[M, e_i] = 0, \quad \quad \quad \quad  \quad \quad \left\{e_0, e_i\right\} = e_{i+1}, & \textrm{for $\underline{n_{t-1}} < i\leq \underline{n_t}$ (if $t$ is even),} 
        \end{align*}
        where $M= \textrm{span}(e_0, e_{1})$ is complementary to $\mathcal{L}^2$ in $\mathcal{L}$, $e_i \in \mathcal{L}^i$ for $i>1$.
\end{lem}
\begin{proof}   
    Construct the basis $\{e_0, \ldots, e_n\}$ as follows. Let $M$ be a complementary space to $\mathcal{L}^2$ in $\mathcal{L}$. Then there is some minimal $k \geq 0$ such that $P_M^k(M) = 0$. Set $n_1 = k - 1$. Assume $n_1 \neq 0$, otherwise swap the brackets.   
    There are $e_0, e_1 \in M$ such that $P_{e_0}^{n_1}(e_1)\neq 0$, so set $P_{e_0}^{i}(e_1) = e_{i+1}$ for $1\leq i\leq n_1$. Moreover,  then $[M, e_{n_1+1}] = 0$.

    Next, if $n_1= n$, then the stated result for $t=1$ is obtained. If $n_1<n$, it follows $Q_{M}(e_{n_1+1})\neq 0$, since $\mathcal{L}$ is filiform. 
   If there exists some minimal $k \geq 1$ such that $P_M Q_M^{k}(e_{n_1+1})\neq 0$ and there is no loss of generality in assuming that $P_{e_0} Q_{e_0}^{k}(e_{n_1+1})\neq 0$ for certain $e_0\in M$. Choose $n_2 = k$ and for $1\leq i\leq n_2$ set $Q_{e_0}^{i}(e_{n_1+1}) = e_{{n_1+1}+i}$. If $P_M Q_M^{k}(e_{n_1+1})=0$ for all $k\in \mathbb{N}$, then $t=2$, we have $ Q_M^{k}(e_{n_1+1}) \neq 0$ for $k \leq n - n_1$ and the basis is clear, because $\mathcal{L}$ is filiform.

   Now, the result follows by induction. In short, suppose the argument above was applied $m$ times. Then there is an element $e_{l}$ such that $P_{M}(e_{l}) \neq 0$. Clearly, there is some minimal $k \geq 0$ such that $P_M^k(e_l) = 0$ by the nilpotency, and we can proceed as in the first paragraph. If after that, the basis is complete, we are done. Otherwise, proceed as in the second paragraph.    Note that this is indeed a basis because of the nilpotency of the compatible algebra.
\end{proof}

Observe that in the previous result, the integers $n_{2i}$ and $t$ are chosen to be minimal, while the values of  $n_{2i-1}$ are taken to be maximal. Although this selection is arbitrary, we adopt it as a convention.

\begin{exa}
Fix a series of natural numbers $s = (n_1, \ldots, n_t)$. The model filiform compatible Lie algebra $(\mathcal{L}_s, [-, -], \left\{-, -\right\})$ associated to the series $s$ is defined for a basis $e_0, \ldots e_n$, where $\underline{n_{j}} = \sum_{i=1}^{j} n_i$ and $n= \underline{n_t}$, by the following products (and with all the other products equal to zero)
        \begin{align*}
            &[e_0, e_i] = e_{i+1}, & \textrm{for $1\leq i\leq \underline{n_1}$,} \\
            &\! \left\{e_0, e_i\right\} = e_{i+1}, & \textrm{for $\underline{n_1} < i\leq \underline{n_2}$,} \\
            &[e_0, e_i] = e_{i+1}, & \textrm{for $\underline{n_2}< i\leq \underline{n_3}$,}\\
            & \ldots & \ldots\\
            &[e_0, e_i] = e_{i+1}, & \textrm{for $\underline{n_{t-1}} < i\leq \underline{n_t}$ (if $t$ is odd),} \\
            &\!\left\{e_0, e_i\right\} = e_{i+1}, & \textrm{for $\underline{n_{t-1}} < i\leq \underline{n_t}$ (if $t$ is even),} 
        \end{align*}

Verifying that the algebra $\mathcal{L}_s$ is indeed a filiform compatible Lie algebra is straightforward. Clearly, if $s= (n)$ or $s = (0, n)$ we obtain the model filiform Lie algebra. Notice that the numbers $ n_1, \ldots, n_k $ correspond to the steps where the bracket switches between the two operations.
\end{exa}

The cohomology for compatible Lie algebras is defined analogously to the Lie algebra cohomology for a pair of cocycles $(\varphi, \psi)$, taking into account the compatibility. We briefly recall the definition of the second cohomology space with respect to the adjoint representation for convenience. 
Let $(\mathcal{L}, [-, -], \left\{-, -\right\})$ be a compatible Lie algebra. Denote $\mathcal{L}_1 = (\mathcal{L}, [-, -])$ and $\mathcal{L}_2 = (\mathcal{L}, \{-, -\})$. The space of $2$-cocycles of $\mathcal{L}$ is
$$Z^2_{CL}(\mathcal{L}, \mathcal{L}) \subseteq Z^2_{Lie}(\mathcal{L}_1, \mathcal{L}_1) \times  Z^2_{Lie}(\mathcal{L}_2, \mathcal{L}_2)$$
with the additional condition, for $x, y, z \in \mathcal{L}$ and $(\varphi, \psi)\in Z^2_{CL}(\mathcal{L}, \mathcal{L})$, that follows from the compatibility 
{\begin{equation}\label{compcoho}
    \begin{aligned}
&\varphi(x, \{y, z\}) + \varphi(y, \{z, x\}) + \varphi(z, \{x, y\}) + \{x, \varphi(y, z)\} + \{y, \varphi(z, x)\} + \{z, \varphi(x, y)\}  \\
 & \quad \quad  + \psi(x, [y, z]) + \psi(y, [z, x]) + \psi(z, [x, y]) + [x, \psi(y, z)] + [y, \psi(z, x)] + [z, \psi(x, y)]  = 0.
\end{aligned}
\end{equation}}
The next result is obtained easily by an algebraic verification.

\begin{lem}\label{lemco1}
    Let $(\mathcal{L}, \alpha, \beta)$ be a compatible Lie algebra. Then
    \medskip  \begin{enumerate}
        \item $(\lambda_1 \alpha + \lambda_2 \beta, \mu_1 \alpha + \mu_2 \beta)\in Z^2_{CL}(\mathcal{L}, \mathcal{L})$  for all $\lambda_i, \mu_j \in \mathbb{C}$.
        \medskip 
        \item $(\mathcal{L}, \lambda_1 \alpha + \lambda_2 \beta, \mu_1 \alpha + \mu_2 \beta)$ is a compatible Lie algebra  for all  $\lambda_i, \mu_j \in \mathbb{C}$.
    \end{enumerate}
\end{lem}

\begin{prop}
    The compatible Lie algebra $\mathcal{L}_s$ for $s =  (n_1, \ldots, n_t)$ is a linear deformation of the model filiform Lie algebra $\mathcal{L}_k$ where $k = \sum_{i=1}^{j} n_i$.
\end{prop}
\begin{proof}
    Let $(\mathcal{L}_s, \alpha, \beta)$ be the model filiform compatible Lie algebra for $s =  (n_1, \ldots, n_t)$. The algebra $(\mathcal{L}, \alpha + \beta, 0)$ is the model filiform Lie algebra $\mathcal{L}_k$, with $k$ as in the statement.     
    It suffices to show that $(\beta, - \beta)$ is a $2$-cocycle for the compatible Lie algebra $(\mathcal{L}, \alpha + \beta, 0)$ and that $(\mathcal{L}, \beta, - \beta)$ is a compatible Lie algebra. But this is clear, since $\beta$ is a $2$-cocycle of the Lie algebra $(\mathcal{L}, \alpha + \beta)$ and $(\mathcal{L}, \beta)$ is a Lie algebra.    
\end{proof}

\begin{prop}
    Let $(\mathcal{L}, \alpha, \beta)$ be a nilpotent compatible Lie algebra with nilindex $k$. Then for all  $\lambda_i, \mu_j \in \mathbb{C}$ the compatible Lie algebra $(\mathcal{L}, \lambda_1 \alpha + \lambda_2 \beta, \mu_1 \alpha + \mu_2 \beta)$ is nilpotent too with nilindex at most $k$.
\end{prop}
\begin{proof}
    The result follows by the fact that there is a $k\in \mathcal{N}$ such that the monomials with all the combinations of brackets $\alpha$ and $\beta$ of length $k$ is zero. Since the monomials with any linear combinations of the brackets $\lambda_1 \alpha + \lambda_2 \beta, \mu_1 \alpha + \mu_2 \beta$ can be linearly expanded to a monomial of the previous type, we conclude that the algebra $(\mathcal{L}, \lambda_1 \alpha + \lambda_2 \beta, \mu_1 \alpha + \mu_2 \beta)$ must be nilpotent of length at most $k$.
\end{proof}

    In fact, there are infinitely many $\mu \in \mathbb{C}$ such that $\mathcal{L}' = (\mathcal{L}, \alpha, \beta + \mu \alpha)$ is nilpotent of nilindex $k$. Moreover, if $(\mathcal{L}, \alpha, \beta)$ is a filiform compatible Lie algebra, we can choose $\mu$ such that $\beta + \mu \alpha$ is a filiform Lie algebra. Thus $\mathcal{L}$ can be identified with a pair $(f, \alpha')$, where $f := \beta + \mu \alpha$ is a filiform Lie algebra and $\alpha' := \mu \alpha$ is a $2$-cocycle in $Z^2_{Lie}(\mathcal{L}, \mathcal{L})$ that defines a linear deformation of $f$. This suggests the generalization of a theorem by Vergne \cite{Vergne}.

\begin{conjecture}
    Let $\mathcal{L}$ be a filiform compatible Lie algebra of dimension $n+1$. Then $\mathcal{L}$ is a linear deformation of the algebra $\mathcal{L}_s$ for certain $s \in \mathbb{N}^t$.
\end{conjecture}

\subsection{Graded filiform compatible Lie algebras}

Let $\mathcal{L}$ be a filiform compatible Lie algebra. Consider the natural descending filtration obtained by the ideals $\mathcal{L}^i$. Then we can associate to $\mathcal{L}$ a graded algebra $\textrm{gr}\, \mathcal{L}$ by setting $\textrm{gr}\, \mathcal{L} = \oplus \, \mathcal{L}_i$ where $ \mathcal{L}_i = \mathcal{L}^{i-1}/\mathcal{L}^i$ with products defined for homogeneous elements by
$$
\begin{tabular}{llllll}
$[-,-]_{ij} : $ &$ \mathcal{L}_i \times \mathcal{L}_j$ &$\to \mathcal{L}_{i+j}$ 
& 
$\{-, -\}_{ij} : $ &$ \mathcal{L}_i \times \mathcal{L}_j$ &$\to \mathcal{L}_{i+j}$ 
\\[1ex]
 &$(x + \mathcal{L}^i,\, y + \mathcal{L}^j)$ &$\mapsto [x, y] + \mathcal{L}^{i+j}$ 
& 
&$(x + \mathcal{L}^i,\, y + \mathcal{L}^j)$ &$\mapsto \{x, y\} + \mathcal{L}^{i+j}$ 
\end{tabular}
$$
    The algebra $\textrm{gr}\, \mathcal{L}$ introduced above is a graded compatible Lie algebra. Moreover, $\textrm{gr}\, \mathcal{L}$ is a filiform algebra by Lemma~\ref{lemmafili1}.

\begin{prop} \label{grL}
    Let $\mathcal{L}$ be a filiform compatible Lie algebra of dimension $n+1$. Then, there is a homogeneous basis $\left\{X_0, \ldots, X_n\right\}$ of $\textrm{gr}\, \mathcal{L}$ such that
  { $$\begin{array}{ll}
        X_0, X_1 \in \mathcal{L}_1, & X_i\in  \mathcal{L}_i \quad i=2,\dots,n_1+1, \\{}
        [X_0,X_i]=X_{i+1}, \ 1 \leq i \leq n_1, & [X_0,X_{n_1+1}]=0 \\{}
        [X_1,X_2]=0, & [X_i,X_j]=0 \ 1\leq i <j\leq n_1+1 \mbox{ and } i+j\neq n_1+1,\\{}
        [X_i,X_{n_1+1-i}]=(-1)^{i}\alpha X_{n_1+1}& \mbox{ with } \alpha=0 \mbox{ if } n_1 \mbox{ is odd}
    \end{array}$$}
    
\end{prop}
\begin{proof}
    {Suppose $\mathcal{L}$ as in Lemma \ref{lemmafili1}, i.e. 
  \begin{align*}
            &[e_0, e_i] = e_{i+1}, & \textrm{for $1\leq i\leq \underline{n_1}$,} \\
            &[M, e_i] = 0, \quad \quad \quad \quad  \quad \quad \left\{e_0, e_i\right\} = e_{i+1}, & \textrm{for $\underline{n_1} < i\leq \underline{n_2}$,} \\
            &[e_0, e_i] = e_{i+1}, & \textrm{for $\underline{n_2}< i\leq \underline{n_3}$,}\\
            & \ldots & \ldots\\
            &[e_0, e_i] = e_{i+1}, & \textrm{for $\underline{n_{t-1}} < i\leq \underline{n_t}$ (if $t$ is odd),} \\
            &[M, e_i] = 0, \quad \quad \quad \quad  \quad \quad \left\{e_0, e_i\right\} = e_{i+1}, & \textrm{for $\underline{n_{t-1}} < i\leq \underline{n_t}$ (if $t$ is even),} 
        \end{align*}
        where $M= \textrm{span}(e_0, e_{1})$ is complementary to $\mathcal{L}^2$ in $\mathcal{L}$, $e_i \in \mathcal{L}^i$ for $i>1$. 

        First, let  $Y_i$ be a nonzero vector belonging to $\textrm{gr}\, \mathcal{L}$ with $1 \leq i \leq n_1 +1$. There is $Y \in \mathcal{L}_1$, such that 
        $$[Y,Y_i]=f_i(Y)Y_{i+1}$$
    with the linear mapping $f_i:\mathcal{L}_1 \longrightarrow \CC $ being not identically null. There is $X_0 \in \mathcal{L}_1$ such that $f_i(X_0)\neq 0, \ \forall i, \ 1 \leq i \leq n_1$. One can choose the vectors $X_i \in \mathcal{L}_i$ with $X_i=\lambda_i Y_i $, so that they satisfy the conditions 
    $$[X_0,X_i]=X_{i+1}, \ 1 \leq i \leq n_1, \ [X_0,X_{n_1+1}]=0.$$
    As we have $dim(\mathcal{L}_1)=2$, then $[X_0,X_2]=X_3$ and $[X_1,X_2]=aX_3$. We can assert  that there exists a unique $X_1\in \mathcal{L}_1$ satisfying $[X_1,X_2]=0$. The others relations are determined by induction. We consider the basis $(X_0, \dots , X_{n_1+1})$ of $\mathcal{L}_1 \oplus \mathcal{L}_2 \oplus \cdots \oplus \mathcal{L}_{n_1+1}$ defined above. Then $(X_0,X_1, \dots, X_{n_1})$ is a basis of the filiform compatible graded algebra $(\mathcal{L}_1 \oplus \mathcal{L}_2 \oplus \cdots \oplus \mathcal{L}_{n_1+1})/\CC X_{n_1+1}$ satisfying the relations of the Proposition. Suppose that $n_1+1$ is odd. We have $[X_i,X_j]=0$ as long as $i+j < n_1$. One put $[X_i,X_{n_1+1-i}]=\alpha_i X_{n_1+1}$. By the Jacobi identity we have that $\alpha_i=-\alpha_{i+1}$ and therefore $\alpha_i=(-1)^{i}\alpha$. This gives us the sought-after basis vectors $X_i$ for $1\leq i \leq n_1+1$. Suppose now $n_1+1$ even. Then, we get 

    $$[X_i,X_{n_1-i}]=(-1)^{i}\alpha X_{n_1} \quad [X_i,X_{n_1+1-i}]=\alpha_{i} X_{n_1+1}$$
    the Jacobi identity leads to $(-1)^{i}\alpha=\alpha_{i+1}+\alpha_i$ and thus, $\alpha_i=(-1)^{2+i}(\alpha_1+(i-1)\alpha)$. Since $\alpha_{\frac{n_1+1}{2}}=0$, we deduce $\alpha_1=(1-\frac{n_1+1}{2})\alpha$ and $\alpha_{i}=(-1)^{i+1}(1-\frac{n_1+1}{2})\alpha$. From the Jacobi identity for the triplet $(X_1, X_{\frac{n_1-1}{2}}, X_{n_1-1})$ we have $\alpha=0=\alpha_i$. }
\end{proof}

\begin{thm} Any complex seven dimensional naturally graded filiform compatible Lie algebra with $n_1=4$ is isomorphic to one of the following compatible lie algebras which can be expressed on a certain basis $\{ X_0,$ $X_1,$ $X_2,$ $X_3,$ $X_4,$ $X_5,$ $X_6\}$ into two sub-families as follows:
$$F_1:\left\{
\begin{array}{llll}
   [X_0,X_1]=X_2,  & \{X_0,X_1\}=\alpha_1 X_2,& \ &\  \\{}
  [X_0,X_2]=X_3,  & \{X_0,X_2\}=\alpha_2 X_3,& &  \\{}
  [X_0,X_3]=X_4,  & \{X_0,X_3\}=\alpha_3 X_4,& &  \\{}
[X_0,X_4]= X_5,  & \{X_0,X_4\}=\alpha_4 X_5,& & \\{}
  & \{X_0,X_5\}=X_6.& & \\
\end{array}\right.$$
$$F_2:\left\{
\begin{array}{llll}
   [X_0,X_1]=X_2,  & \{X_0,X_1\}=\alpha X_2,& \{X_1,X_2\}=\lambda X_3,&\{X_2,X_3\}=\lambda X_5, \\{}
  [X_0,X_2]=X_3,  & \{X_0,X_2\}=\beta X_3,&\{X_1,X_3\}=\lambda X_4, & \lambda \neq 0 \\{}
  [X_0,X_3]=X_4,  & \{X_0,X_3\}=\beta X_4,& &  \\{}
[X_0,X_4]= X_5,  & \{X_0,X_4\}=\alpha X_5,& & \\{}
  & \{X_0,X_5\}=X_6.& & \\
\end{array}\right.$$

\end{thm}

\begin{proof} On account of Proposition \ref{grL} we have the following non-null bracket products with $\d \in \{0,1\}$
$$\begin{array}{llll}
   [X_0,X_1]=X_2,  & \{X_0,X_1\}=a_{01}X_2,& \{X_1,X_2\}=a_{12}X_3,&\{X_2,X_3\}=a_{23}X_5,  \\{}
  [X_0,X_2]=X_3,  & \{X_0,X_2\}=a_{02}X_3,& \{X_1,X_3\}=a_{13}X_4,&\{X_2,X_4\}=a_{24}X_6,  \\{}
  [X_0,X_3]=X_4,  & \{X_0,X_3\}=a_{03}X_4,& \{X_1,X_4\}=a_{14}X_5,&  \\{}
[X_0,X_4]= X_5,  & \{X_0,X_4\}=a_{04}X_5,& & \\{}
[X_1,X_4]= -\d X_5,  & & & \\{}
[X_2,X_3]= \d X_5,  & & & \\{}
  & \{X_0,X_5\}=X_6,&\{X_1,X_5\}=a_{15}X_6, &\{X_2,X_4\}=a_{24}X_6. \\
\end{array}$$
From the mixed Jacobi identity for the triplets $(X_0,X_1,X_i)$ with $2\leq i \leq 4$, $(X_1,X_2,X_3)$ and $(X_0,X_2,X_3)$ we have the following restrictions
$$\begin{array}{l}
    (1) \quad a_{12}=a_{13} \\
    (2) \quad a_{23}+a_{14}-a_{12}+\d (a_{01}-a_{03})=0 \\
    (3) \quad a_{24}+a_{15}+\d =0\\
    (4) \quad \d a_{15}=0 \\
    (5) \quad a_{24}- \d =0
\end{array}
    $$
If $\d =1$ then by $(5)$ we have $a_{24}=1$ and from $(3)$ we deduce $a_{15}=-2$ which is a contradiction with $(4)$. Therefore, $\d =0$ and then, by $(5)$ and $(3)$ it is obtained $a_{24}=a_{15}=0$. Next, from the Jacobi identity on the bracket product $\{ -, -\}$ applied on the triplets $(X_0,X_1,X_4), \ (X_0,X_1,X_2)$ and $(X_0,X_1,X_3)$ we get 
  $$\begin{array}{l}
    (6) \quad a_{14}=0 \\
    (7) \quad a_{02}a_{12}-a_{12}a_{03}=0 \\
    (8) \quad a_{01}a_{23}-a_{12}a_{04}=0\\
    \end{array}
    $$  
    also $a_{12}=a_{13}=a_{23}$. Thus we have two possibilities, if $a_{12}=\lambda=0$ then we have four free parameters $\alpha_1=a_{01}, \ \alpha_2=a_{02}, \ \alpha_3=a_{03}$ and $\alpha_4=a_{04}$. On the other hand, if $a_{12}=\lambda\neq 0$ then we have only two free parameters $\alpha=a_{01}=a_{04}$ and $\beta=a_{02}=a_{03}$. 
\end{proof}

\subsection{Solvable extensions of the algebra $\cal L_s$.}

Following Section~\ref{sec3}, we study the existence of solvable Lie algebras whose nilradical is the filiform compatible Lie algebra $\mathcal{L}_s$. We begin by describing the algebra of  derivations of $\mathcal{L}_s$, which will be the key input to prove the existence of solvable extensions. 
The next observation will be used to describe the algebra of derivations of $\mathcal{L}_s$. 
Let $(\gg,[\,,\,]_1,[\,,\,]_2)$ be a compatible Lie algebra and set for all $x,y \in \gg$ the bracket $[[x,y]]=[x,y]_1+[x,y]_2$. 
Any derivation for both $[\,,\,]_1$ and $[\,,\,]_2$ is a derivation of $(\gg,[[\,,\,]])$. 
Therefore the algebra of derivations of $(\gg,[\,,\,]_1,[\,,\,]_2)$ is a subalgebra of the algebra of derivations of $(\gg,[[\,,\,]])$. 
For $\mathcal{L}_s$, the sum bracket $[[\,,\,]]$ is the model filiform Lie algebra, hence every derivation of $\mathcal{L}_s$ is a derivation of the model filiform Lie algebra $\mathcal{L}_k$, see Remark~\ref{rmderL}. 

\begin{rem}
    To exploit this fact, we introduce the following general construction. Let $\cal N$ be a $n$-dimensional nilpotent compatible Lie algebra with Lie components $\cal N_1$ and $\cal N_2$ such that $\textrm{Der}(\cal N)$ has a maximal torus $\mathcal{T}$ of diagonal derivations.   Choose two sets $D=\{d_1, \dots, d_r\}$ and $D'=\{d_1', \dots, d_s'\}$ of linearly independent derivations in $\mathcal{T}$.
    Set
\[
\mathcal R=\mathcal N\oplus \mathcal Q,\qquad 
\mathcal Q=\operatorname{span}\{x_1,\dots,x_r,y_1,\dots,y_s\},
\]
with $[[\mathcal Q,\mathcal Q]]=0$ and actions on $\mathcal N$ given by
\[
P_{x_i}\!\mid_{\mathcal N}=d_i,\quad Q_{x_i}\!\mid_{\mathcal N}=0\quad(1\le i\le r),\qquad
P_{y_j}\!\mid_{\mathcal N}=0,\quad Q_{y_j}\!\mid_{\mathcal N}=d'_j\quad(1\le j\le s).
\]
Then $\mathcal R$ is a compatible solvable Lie algebra of dimension $n+r+s$ with nilradical $\mathcal N$.
\end{rem}

At this point, we recall that the maximal torus of the model filiform Lie algebra has dimension two. 

\begin{thm} \label{thm5.1} For $1\le k\le 4$, there exists a solvable compatible Lie extension of $\cal L_s$ of dimension $(\underline{n_t}+1+k)$.
\end{thm}
\begin{proof}
    This follows from the preceding construction and a direct verification that every diagonal derivation of the model filiform Lie algebra is a derivation of $\mathcal L_s$.
\end{proof}

We close the paper with another open question, suggested by the solvable extensions constructed above.

\begin{conjecture}
    There is no more than $\underline{n_s}+5$ dimensional solvable compatible Lie extensions of $\cal L_s$.
\end{conjecture}

\bigskip

\section*{Declarations}

\subsection*{Funding}
{The authors are partially supported by Uni\'on Europea,  Fondo Europeo de Desarrollo Regional and by Junta de Extremadura, la Autoridad de Gestión y el Ministerio de Hacienda, through project GR24068, and also  by Agencia Estatal de Investigaci\'on (Spain), grant PID2024-155502NB-I00 (European FEDER support included, EU)}

\subsection*{Data availability statement}
No datasets were generated or analyzed during the preparation of this article; therefore, data sharing is not applicable.

\subsection*{Conflict of interest}
The authors declare that they have no conflicts of interest. No issues related to ethical approval, conflicts of interest, or ethical standards arise in connection with this work.

\bigskip 

\bibliographystyle{amsplain}

\end{document}